\documentclass[12pt,reqno]{amsart}
\usepackage{amsmath,amsthm,amsfonts,amssymb,amscd,amstext}
\usepackage[ansinew]{inputenc}
\usepackage[dvips]{graphicx}
\usepackage{psfrag,euscript}

\usepackage{pxfonts}   
\usepackage{mathpazo} 

\usepackage{pslatex}
\numberwithin{equation}{section}

\theoremstyle{plain}
\newtheorem{maintheorem}{Theorem}

\newtheorem{theorem}{Theorem}[section]
\newtheorem{proposition}[theorem]{Proposition}
\newtheorem{corollary}[theorem]{Corollary}
\newtheorem{lemma}[theorem]{Lemma}

\theoremstyle{definition}
\newtheorem{remark}[theorem]{Remark}

\newcommand{\RR}{{\mathbb R}}

\newcommand{\NN}{{\mathbb N}}
\newcommand{\ZZ}{{\mathbb Z}}

\newcommand{\diam}{\operatorname{diam}}
\renewcommand{\epsilon}{\varepsilon}
\newcommand{\dist}{\operatorname{dist}}

\newcommand{\per}{\operatorname{Per}}

\newcommand{\supp}{\operatorname{supp}}

\newcommand{\cT}{\EuScript{T}}

\newcommand{\cP}{\EuScript{P}}
\newcommand{\cO}{\EuScript{O}}

\newcommand{\U}{\EuScript{U}}
\newcommand{\G}{\EuScript{G}}

\newcommand{\R}{\EuScript{R}}
\newcommand{\Z}{\EuScript{Z}}

\def \cA {{\mathcal A}}

\def \cD {{\mathcal D}}

\def \cH {{\mathcal H}}

\def \cO {{\mathcal O}}
\def \cP {{\mathcal P}}

\def \cR {{\mathcal R}}

\def \cT {{\mathcal T}}
\def \cU {{\mathcal U}}

\def \NN {{\mathbb N}}

\def \RR {{\mathbb R}}

\def \ZZ {{\mathbb Z}}

\def \fX {{\mathfrak X}}

\def \fD {{\mathfrak D}}

\title [Dominated splitting and zero volume]
{Dominated splitting and zero volume for
  incompressible three-flows}

\author{Vitor Araujo}
\email{vitor.araujo@im.ufrj.br \text{and} vdaraujo@fc.up.pt}
\address{
Instituto de Matem\'atica,
Universidade Federal do Rio de Janeiro
\\
C. P. 68.530,
  21.945-970 Rio de Janeiro, RJ-Brazil 
\\
\emph{and} 
\\
Centro de Matem\'atica da Universidade do Porto
\\
Rua do Campo Alegre 687, 
4169-007 Porto, Portugal.}

\author{M\'ario Bessa}
\email{bessa@fc.up.pt \text{and} bessa@impa.br}
\address{Centro de Matem\'atica da Universidade do Porto
\\
Rua do Campo Alegre 687, 
4169-007 Porto, Portugal.}

\thanks{V.A. was partially supported by CNPq, FAPERJ and
  PRONEX-Dynamical Systems (Brazil) and FCT (Portugal)
  through CMUP and POCI/MAT/61237/2004. M.B. was partially
  supported by FCT-FSE, SFRH/BPD/20890/2004.}

\date{\today}

\begin{document}

\begin{abstract}
  We prove that there exists an open and dense subset of the
  incompressible $3$-flows of class $C^2$ such that, if a
  flow in this set has a positive volume regular invariant
  subset with dominated splitting for the linear Poincar\'e
  flow, then it must be an Anosov flow. With this result we
  are able to extend the dichotomies of Bochi-Ma\~n\'e (see
  \cite{Mane96,Bochi02,bessa}) and of Newhouse
  (see~\cite{Newhouse77,BesDu07}) for flows \emph{with}
  singularities. That is we obtain for a residual subset of
  the $C^1$ incompressible flows on $3$-manifolds that: (i)
  either all Lyapunov exponents are zero or the flow is
  Anosov, and (ii) either the flow is Anosov or else the
  elliptic periodic points are dense in the manifold.
\end{abstract}

\maketitle

\noindent
\textbf{Keywords:} generic incompressible flows, Lyapunov
exponents, dominated splitting, hy\-per\-bo\-li\-ci\-ty.

\bigskip

\noindent
\textbf{\textup{2000} Mathematics Subject Classification:}
Primary: 37D25; Secondary:37D35, 37D50, 37C40.

\tableofcontents

\section{Introduction}
\label{sec:introd}

Incompressible flows are a traditional subject from Fluid
Mechanics, see e.g.~\cite{galla02}.  These flows are
associated to divergence-free vector fields, they preserve a
volume form on the ambient manifold and thus come equipped
with a natural invariant measure. On compact manifolds this
provides an invariant probability giving positive measure
(volume) to all nonempty open subsets. Therefore for vector
fields $X$ in this class we have $\Omega(X)=M$ by the
Poincar\'e Recurrence Theorem, where $\Omega(X)$ denotes the
non-wandering set. In particular such flows can have neither
sinks nor sources, and in general do not admit Lyapunov
stable sets, either for the flow itself or for the time
reversed flow.

Let $\fX^r(M)$ be the space of $C^r$ vector fields, for any
$r\ge1$, and $\fX^r_\mu(M)$ the subset of divergence-free
vector fields defining incompressible (or conservative)
flows. It is natural to study these flows under the measure
theoretic point of view, besides the geometrical one.

The device of Poincar\'e sections has been used extensively
to reduce several problems arising naturally in the setting
of flows to lower dimensional questions about the behavior
of a transformation. Recent breakthroughs on the
understanding of generic volume-preserving diffeomorphisms
on surfaces have non-trivial consequences for the dynamics
of generic incompressible flows on three-dimensional
manifolds.

The Bochi-Ma\~n\'e Theorem~\cite{Bochi02} asserts that, for
a $C^1$ residual subset of area-preserving diffeomorphisms,
either the transformation is Anosov (i.e. globally
hyperbolic), or the Lyapunov exponents are zero Lebesgue
almost everywhere (i.e. there is no asymptotic growth of the
length of vectors in any direction for almost all
points). This was announced by Ma\~n\'e in~\cite{Mane84} but
only a sketch of a proof was available in~\cite{Mane96}. The
complete proof presented by Jairo Bochi in~\cite{Bochi02}
admits extensions to higher dimensions, obtained by Bochi
and Viana in~\cite{BoV02}, stating in particular that either
the Lyapunov exponents of a $C^1$ generic volume-preserving
diffeomorphism are zero Lebesgue almost everywhere, or else
the system admits a dominated splitting for the tangent
bundle dynamics. A survey of this theory can be found
in~\cite{BochVi04}.

Recently (see Theorem~\ref{thm:bessa1} below) one of the
coauthors was able to use, adapt and fully extend the ideas
of the original proof by Bochi to the setting of generic
conservative flows on three-dimensional compact boundaryless
manifolds \emph{without singularities}, in~\cite{bessa}.
The presence of singularities imposes some differences
between discrete and continuous systems. The ideas from the
Bochi-Mañé proof were partially extended to a \emph{dense}
subset of $C^1$ incompressible flows (see
Theorem~\ref{thm:bessa2} below) admitting singularities but
\emph{without a full dichotomy} between zero exponents and
global hyperbolicity in the same work~\cite{bessa}.

There are related results from Arbieto-Matheus
in~\cite{ArbiMath}, where it is proved that $C^1$ robustly
transitive volume-preserving 3-flows must be Anosov, with
the help of a new perturbation lemma for divergence-free
vector fields, and also from Horita-Tahzibi
in~\cite{hortah06}, where it is proved that robustly
transitive symplectomorphisms must be partially hyperbolic. One of the coauthors together with Rocha proved in~\cite{BeRo2007} that robustly
transitive volume-preserving $n$-flows must have dominated splitting.

There are older $C^1$ dichotomy results for low dimensional
transformations.  A result of fundamental importance in the
theory of generic conservative diffeomorphisms on surfaces
was obtained by Newhouse in~\cite{Newhouse77}.  Newhouse's
theorem states that $C^1$ generic area-preserving
diffeomorphisms on surfaces either are Anosov, or else the
elliptical periodic points are dense.  A refined version of
this results was presented by Arnaud in~\cite{MCA02} in the
family of 4-dimensional symplectomorphisms. Even more
recently Saghin-Xia~\cite{SagXia06a} generalized Arnauld
result for the multidimensional symplectic case, and
in~\cite{BesDu07} one of the coauthors together with Duarte
obtained a similar dichotomy for $C^1$-generic
incompressible flows \emph{without singularities} on
$3$-manifolds: either the flow is Anosov, or else the
elliptic periodic orbits are dense in the manifold.

Here we complete the results of~\cite{bessa}
and~\cite{BesDu07} fully extending the dichotomy from
generic \emph{non-singular} vector fields to generic vector
fields in the family of $C^1$ \emph{all} incompressible
flows on $3$-manifolds.

The main step is our arguments is to show that if a $C^2$
incompressible flow on a $3$-manifold admits a positive
volume invariant subset (not necessarily closed) formed by
regular orbits with a very weak form of hyperbolicity, known
as \emph{dominated decomposition}, then there cannot be any
singularity on the closure of this set, under a mild
non-resonant conditions on the possible eigenvalues at the
singularities. This leads easily to the conclusion that the
closure of this invariant subset is a positive volume
hyperbolic subset. 

Adapting arguments from Bochi-Viana~\cite{BochVi04} to the
flow setting it is proved that incompressible $C^2$ flows
with positive volume compact invariant hyperbolic sets must
be globally hyperbolic. Finally using standard arguments
from Bochi-Viana~\cite{BochVi04},~\cite{bessa} and \cite{BesDu07} these
results imply the $C^1$ generic dichotomies mentioned above
for incompressible flows without any extra condition on the
singularities.

\subsection{Definitions and statement of the results}
\label{sec:statem-results}

In what follows $M$ will always be a $C^\infty$ compact
connected boundaryless three-dimensional Riemannian
manifold. We denote by $\mu$ a volume form on $M$ and by
$\dist$ the distance induced on $M$ by the Riemannian scalar
product, denoted by $<\cdot,\cdot>$.

We begin by recalling Oseledets' Theorem for measure
preserving flows and the notion of \emph{Linear
  Poincar\'e Flow} first introduced by Doering
in~\cite{Do87}.

Consider $X\in{\mathfrak{X}_{\mu}^{1}}(M)$ and the
associated flow $X^{t}:M\rightarrow{M}$. Oseledets' Theorem
~\cite{Os68} guarantees that we have, for $\mu$-a.e. point
$x\in{M}$, a measurable splitting of the tangent bundle at
$x$, $T_{x}M=E^{1}_{x}\oplus...\oplus{E^{k(x)}_{x}}$, called
the \emph{Oseledets splitting} and real numbers
$\lambda_{1}(x)>...>\lambda_{k(x)}(x)$ called \emph{Lyapunov
  exponents} such that
$DX^{t}_{x}(E^{i}_{x})=E^{i}_{X^{t}(x)}$ and
$$
\lim_{t\to\pm\infty}\frac1{t}\log{\|DX^{t}_{x}\cdot
  v^{i}\|=\lambda_{i}(x)}
$$
for any $v^{i}\in{E^{i}_{x}\setminus\vec{0}}$ and
$i=1,...,k(x)$. Oseledets' Theorem allow us to conclude also
that
\begin{equation}\label{angle}
\lim_{t\rightarrow{\pm{\infty}}}\frac{1}{t}\log{|\det(DX^{t}_{x})|
=\sum_{i=1}^{k(x)}\lambda_{i}(x).\dim(E^{i}_{x})}
\end{equation}
which is related to the sub-exponential decrease of the
angle between any subspaces of the Oseledets splitting along
$\mu$-a.e.  orbits. Since $DX^{t}_{x}(X(x))=X(X^{t}(x))$ the
direction of the vector field is one of the Oseledets
subspaces and it is associated to a zero Lyapunov
exponent. The full $\mu$-measure subset of points where
these exponents and directions are defined will be referred
to as the set of \emph{Oseledets points} of $X$.

In the volume-preserving setting we have
$|\det(DX^{t}_{x})|=1$. Hence on 3-manifolds
by~(\ref{angle}) either $\lambda_{1}(x)=-\lambda_{3}(x)>0$
or both are zero. If $\lambda_{1}(x)>0$, then we obtain two
directions $E^{u}_{x}$ and $E^{s}_{x}$ respectively
associated to $\lambda_{1}(x)$ and $\lambda_{3}(x)$ which we
denote by $\lambda_{u}(x)$ and $\lambda_{s}(x)$.

We say that $\sigma\in M$ is a \emph{singularity} of $X$ if
$X(\sigma)=\vec{0}$ and we denote by $S(X)$ the set of all
singularities of $X$. The complement $M\setminus S(X)$ is
the set of \emph{regular} points for the flow of $X$.
For a regular point $z$ of $X$ denote by
\begin{align*}
  N_z=\{v\in T_zM: <v, X(z)>=0\}
\end{align*}
the orthogonal complement of the flow direction
$[X]_z=[X(z)]:=\RR\cdot X(z)$ in $T_zM$.  Denote by
$O_z:T_zM\to N_z$ the orthogonal projection of $T_z M$ onto
$N_z$.  For every $t\in\RR$ define
\begin{align*}
P_X^t(z):N_z\to N_{X^t(z)}
\quad\text{by}\quad
P_X^t(z)=O_{X^t(z)}\circ DX^t_z.
\end{align*}
It is easy to see that $P=\{P_X^t(z):t\in\RR, X(z)\neq \vec{0}\}$
satisfies the cocycle identity
\begin{align*}
  P^{s+t}_{X}(z)=P^s_{X}(X^t(z))\circ P^s_{X}(z)
  \quad\text{for every}\quad t,s\in\RR.
\end{align*}
The family $P$ is called the {\em Linear Poincar\'e Flow} of
$X$.

If we have an Oseledets point $x\notin S(X)$ and
$\lambda_{1}(x)>0$, the Oseledets splitting on $T_{x}M$
induces a $P^{t}_{X}$-invariant splitting on $N_{x}$, say
$O_{x}(E^{\diamond}_{x})=N^{\diamond}_{x}$ for
$\diamond=u,s$. If $\lambda_{1}(x)=0$, then the
$P_{X}^{t}$-invariant splitting is trivial. Using
(\ref{angle}) it is easy to see that the Lyapunov exponents
of $P_{X}^{t}(x)$ associated to the subspaces $N^{u}_{x}$
and $N^{s}_{x}$ are respectively $\lambda_{u}(x)\geq 0$ and
$\lambda_{s}(x)\leq{0}$.

We now define dominated structures for the Linear
Poincar\'e Flow. Given a \emph{regular} invariant subset
$\Lambda$ for $X\in\fX^1(M)$, that is $\Lambda\cap
S(X)=\emptyset$, an invariant splitting $N^1\oplus N^2$ of
the normal bundle $N_\Lambda$ for the Linear Poincar\'e Flow
$P^t_X$ is said to be \emph{$m$-dominated}, if there exists
an integer $m$ such that for \emph{every} $x\in\Lambda$ we
have the domination relation
\begin{equation}\label{DS}
  \frac{\big\| P^m_{X}(x)\mid N^1\big\|}{\big\| P^m_{X}(x)\mid
    N^2\big\|} \le \frac12.
\end{equation}
Dominated splittings are automatically continuous on the
Grassmanian of plane subbundles of the tangent bundle, see
e.g. \cite{HPS77,BDV2004} for an exposition of the theory.
In particular the dimensions of the subbundles are constant
on each connected component of $\Lambda$.

As is traditional we say that a vector field is
\emph{Anosov} if the flow preserves a globally defined
hyperbolic structure, that is, the tangent bundle $TM$
splits into three continuous $DX^t$-invariant subbundles
$E\oplus [X]\oplus G$ where $[X]$ is the flow direction, the
sub-bundle $E\neq\vec{0}$ is uniformly contracted and the
sub-bundle $G\neq\vec{0}$ is uniformly expanded by $DX_t$
for $t>0$. Note that for an Anosov flow $X$ the entire
manifold is $m$-dominated for some $m\in\NN$. The fact that
the dimensions of the subbundles are constant on the entire
manifold implies that $S(X)=\emptyset$ for an Anosov vector
field.

Denote by $\fX^r_\mu(M)^\star$ the subset of $\fX^r_\mu(M)$
of $C^r$ incompressible flows but \emph{without
  singularities}. 

\begin{theorem}{\cite[Theorem 1]{bessa}}
  \label{thm:bessa1}
  There exists a residual set $\cR\subset \fX^1_\mu(M)^\star$
  such that, for $X\in\cR$, either $X$ is Anosov or else for
  Lebesgue almost every $p\in M$ all the Lyapunov exponents
  of $X^t$ are zero.
\end{theorem}

Developing the ideas of the proof of this result one can
also obtain the following statement on denseness of
dominated splitting, now admitting singularities.

\begin{theorem}{\cite[Theorem 2]{bessa}}
  \label{thm:bessa2}
  There exists a dense set $\fD\subset \fX^1_\mu(M)$ such
  that for $X\in\fD$, there are invariant subsets $D$ and
  $Z$ whose union has full measure, such that
\begin{itemize}
\item for $p\in Z$ the flow has only zero Lyapunov exponents;
\item $D$ is a countable increasing union $\Lambda_{m_n}$ of
  invariant sets admitting an $m_{n}$-dominated splitting
  for the Linear Poincar\'e Flow, where $m_n$ is a strictly
  increasing integer sequence.
\end{itemize}
\end{theorem}

We recall another $C^1$-type result for
in\-com\-pre\-ssi\-ble three-di\-men\-sio\-nal flows without
fixed points. Preliminary versions for the discrete
symplectic case were presented
in~\cite{Newhouse77,MCA02,SagXia06a} respectively for
surfaces, $4$-dimensional manifolds and $2n$-dimensional
manifolds.

\begin{theorem}\label{thm:BD}{\cite[Theorem 1.2]{BesDu07}}
  Given $\epsilon>0$, any open subset $U$ of $M$ and a non
  Anosov vector field $X\in
  \mathfrak{X}^{1}_{\mu}(M)^{\star}$, there exists $Y\in
  \mathfrak{X}^{1}_{\mu}(M)^{\star}$ such that $Y$ is
  $\epsilon$-$C^1$-close to $X$ and $Y$ has an elliptic
  closed orbit intersecting $U$.
\end{theorem}

We are able to extend Theorem~\ref{thm:bessa1} and
Theorem~\ref{thm:BD} to the full family of incompressible
$C^1$ flows. Here are our main results.

\begin{maintheorem}
  \label{mthm:anosov-zero}
  There exists a generic subset $\R\subset\fX^1_\mu(M)$ such
  that for $X\in\R$
  \begin{itemize}
  \item either $X$ is Anosov,
  \item or else for Lebesgue almost every $p\in M$ all the
    Lyapunov exponents of $X^t$ are zero.
\end{itemize}
\end{maintheorem}

\begin{maintheorem}\label{thr:open}
  Let $\epsilon>0$, an open subset $U$ of $M$ and a non Anosov vector
  field $X\in \mathfrak{X}^{1}_{\mu}(M)$ be
  given. Then there exists $Y\in
  \mathfrak{X}^{1}_{\mu}(M)$ such that $Y$ is
  $C^1$-$\epsilon$-close to $X$ and $Y^{t}$ has an elliptic
  closed orbit intersecting $U$.
\end{maintheorem} 

From Theorem~\ref{thr:open} we can follow \emph{ipsis
  verbis} the proof of Theorem 1.3 of~\cite{BesDu07} to
deduce the next generic result.

\begin{corollary}\label{cor:B}
  There exists a $C^1$ residual set
  $\mathcal{R}\subset\mathfrak{X}^{1}_{\mu}(M)$ such that if
  $X\in\mathcal{R}$, then $X$ is Anosov or else the elliptic
  closed orbits of $X$ are dense in $M$.
\end{corollary}

It is well known that a $C^2$ dynamical system admitting a
hyperbolic set with positive measure must be globally
hyperbolic: see e.g. Bowen-Ruelle~\cite{BR75} and
Bochi-Viana~\cite{BochVi04}. Recently in~\cite{alpi2005}
this was extended to transitive sets having a weaker form of
hyperbolicity called \emph{partial hyperbolicity} with the
extra assumption of non-uniform expansion along the central
direction. Also in~\cite{AAPP} similar results where
obtained for positive volume \emph{singular-hyperbolic} sets
for $C^2$ (not necessarily incompressible) flows.

We extend these results for an even weaker type of
hyperbolicity, i.e. for sets with a dominated splitting.
Both Theorems~\ref{mthm:anosov-zero} and~\ref{thr:open} are
deduced from the following result.

\begin{maintheorem}\label{thm:there-exists-generic}
  There exists an open and dense subset
  $\G\subset\fX^2_\mu(M)$ such that for every $X\in\G$ with
  a regular invariant set $\Lambda$ (not necessarily closed)
  satisfying:
\begin{itemize}
\item the Linear Poincar\'e Flow over $\Lambda$ has a
  dominated decomposition; and
\item $\Lambda$ has positive volume: $\mu(\Lambda)>0$;
\end{itemize}
then $X$ is Anosov and the closure of $\Lambda$ is the whole
of $M$.
\end{maintheorem}

\subsection{Overview of the arguments and organization of
  the paper}
\label{sec:descript-proof}

The proofs of Theorems~\ref{mthm:anosov-zero}
and~\ref{thr:open} follow standard arguments from
Bochi~\cite{AviBoch06},~\cite{bessa} and~\cite{BesDu07}
assuming Theorem~\ref{thm:there-exists-generic} together
with the denseness of $C^2$ incompressible flows among $C^1$
incompressible ones given by Zuppa in~\cite{zup79}.  We
present these arguments in the following
Section~\ref{sec:generic-dichot-incom}.

We give now an outline of the proof of
Theorem~\ref{thm:there-exists-generic}.  Fix $X\in
\fX^2_\mu(M)$ and assume that there exists an invariant
subset $\Lambda$ for $X$ (not necessarily compact) without
singularities (i.e. formed by regular orbits of $X$) and
with positive volume: $\mu(\Lambda)>0$.  We show that
\begin{enumerate}
\item the closure $A$ of
  $\Lambda$ cannot contain singularities.
\end{enumerate}
This is done in Section~\ref{sec:singul-hyperb} combining
arguments from the characterization of robustly transitive
attractors in~\cite{MPP04}, with properties of positive
volume invariant subsets from~\cite{alpi2005} and of
hyperbolic smooth invariant measures from \emph{Pesin's
  Theory}~\cite{Pe77,barreira-pesin2}, together with the
arguments from \cite[Appendix B]{BochVi04}.

\begin{enumerate}
\item[(2)] If $A$ is a compact invariant set without
  singularities and with dominated decomposition of the
  Linear Poincaré Flow, then $A$ is a uniformly hyperbolic
  set.
\end{enumerate}
This is a well-known result from \cite{bessa} and the work of Morales-Pacifico-Pujals in~\cite{MPP04}.

\begin{enumerate}
\item[(3)] a uniformly hyperbolic set $A$ with positive volume
  for a $C^2$ incompressible flow must be the whole $M$.
\end{enumerate}  
For the last item above we adapt the arguments
from~\cite[Appendix B]{BochVi04} to the flow setting.

\subsection*{Acknowledgments} 

V.A.  wishes to thank IMPA for its hospitality, excellent
research atmosphere and access to its superb
library. M.B. wishes to thank CMUP and the Pure Mathematics
Department of University of Porto for access to its
facilities and library during the preparation of this work.


\section{Generic dichotomies for incompressible flows}
\label{sec:generic-dichot-incom}

Here we prove Theorems~\ref{mthm:anosov-zero}
and~\ref{thr:open} assuming
Theorem~\ref{thm:there-exists-generic}.

We start with a sequence of simple lemmas.
We say that the vector field $X$ is \emph{aperiodic} if the
volume of the set of all closed orbits for the corresponding
flow is zero.

\begin{lemma}\label{Bowen2}
  There exists a $C^1$-dense set
  $\cD\subseteq\fX^1_{\mu}(M)$ such that if
  $X\in\mathcal{D}$, then
  \begin{itemize}
  \item $X$ is aperiodic;
  \item $X$ is of class $C^{r}$
    for some $r\geq{2}$;  and
  \item every invariant $m$-dominated set $\Lambda$ has zero
    or full measure, for any $m\in\NN$.
  \end{itemize}
\end{lemma}

\begin{proof} Let $\mathcal{KS}$ be the $C^{r}$ generic
  subset given by~\cite[Theorem 1(i)]{Ro70}, for some
  $r\ge2$, so that $X\in\mathcal{KS}$ is $C^{r}$ and admits
  countably many closed orbits only, all of which are
  hyperbolic or elliptic. 
  According to the results in~\cite{zup79},
  $\fX_{\mu}^{r}(M)$ is also $C^{1}$-dense on
  $\fX_{\mu}^{1}(M)$, for $r\geq 2$. 
  Therefore, we can find a set $\cD$ such that $X\in\cD$ is
  aperiodic, of class at least $C^2$ and given any
  $m$-dominated invariant subset $\Lambda$ of $M$ for $X$,
  by Theorem~\ref{thm:there-exists-generic} we have that
  either $\Lambda$ has zero volume, or $X$ is Anosov, and so
  $\Lambda=M$.
\end{proof}

We define as in~\cite{BoV02} or~\cite{bessa}, the integrated
upper Lyapunov exponent
$$L(X)={\lim_{n\to+\infty}}\int_{M}\frac{1}{n}\log\|P^{n}_{X}(x)\|d\mu(x),$$
which is an upper semicontinuous function $L:\fX_\mu^1(M)\to\RR$. 

The proof of the next result follows~\cite[Proposition
3.2]{bessa} step by step, only replacing hyperbolic
invariant subset with $m$-dominated invariant subset in the
relevant places of the argument.

\begin{proposition}\label{main}
  Let $X\in{\mathfrak{X}_{\mu}^{2}(M)}$ be a
  aperiodic vector field and assume that every
  $m$-dominated invariant subset has zero volume. 

  For every given $\epsilon,\delta>0$ there exists a
  incompressible $C^{1}$ vector field $Y$ such that $Y$ is
  $\epsilon$-$C^1$-close to $X$ and $L(Y)<\delta$.
\end{proposition}

\begin{proof}[Proof of Theorem~\ref{mthm:anosov-zero}]
  Let $\mathcal{D}$ be given by Lemma~\ref{Bowen2}. Denote
  by $\mathcal{A}$ the $C^r$-stable subset of Anosov
  incompressible flows. By upper semicontinuity of $L$, for
  every $k\in{\mathbb{N}}$, the set
  $\cA_{k}=\{ X\in\fX_{\mu}^{1}(M)\colon L(X)<1/k\}$
is open. Then Proposition~\ref{main} implies that
$\mathcal{A}_{k}$ dense in the complement $\cA^{c}$ of $\cA$
in $\fX_{\mu}^1(M)$. We define a $C^{1}$ residual
set by
$$
\cR={\bigcap_{k\in\NN}}\big(\cA\cup \cA_{k}\big).
$$
It is straightforward to check that $\cR$ satisfies the
statement of Theorem~\ref{mthm:anosov-zero}.
\end{proof}

Now we start the proof of Theorem~\ref{thr:open}. But first
we recall a basic result which is a consequence of the
persistence of dominated splittings, see
e.g.~\cite{BDV2004}.

\begin{lemma}\label{le:persistdom}
  Given a subset $\Lambda$ with $m$-dominated splitting for
  a vector field $X$, there exists a neighborhood $U$ of
  $\Lambda$ and $\delta>0$ such that the set
  $\Lambda_{Y}(U):=\cap_{t\in\RR}Y^t(U)$ has a
  $(m+1)$-dominated splitting for any vector field $Y$ which
  is $\delta$-$C^1$-close to $X$.
\end{lemma}

This means that perturbing the original flow $X$ to $Y$
around an invariant $m$-dominated set, we can in \eqref{DS}
switch from $1/2$ to $1/2+\epsilon$ for a very small
$\epsilon$ and for every regular orbit of $Y$ which remains
nearby $\Lambda$.

The following perturbation lemmas from~\cite{BesDu07} are
the main tools in our arguments to prove
Theorem~\ref{thr:open}.

\begin{lemma}{(Small angle perturbation~\cite[Proposition
    3.8]{BesDu07})}
  \label{le:smallangle}
  Let $X\in \fX^{1}_{\mu}(M)$ and $\epsilon>0$ be
  given. There exists $\theta=\theta(\epsilon,X)>0$ such
  that if a hyperbolic periodic orbit $\cO$ for $X$
  has angle between its stable and unstable
  directions smaller than $\theta$, then we can find an
  $\epsilon$-$C^1$-close volume-preserving vector field $Y$
  such that $\cO$ is an elliptic periodic orbit for $Y^t$.
\end{lemma}

Another setting where one can create a nearby elliptic
periodic orbit is the following.

\begin{lemma}{(Large angle perturbation~\cite[Proposition
    3.13]{BesDu07})}
  \label{le:nodomination}
  Let $X\in \fX^{1}_{\mu}(M)$ and $\epsilon,\theta>0$ be
  given. There exists $m=m(\epsilon,\theta)\in\NN$ and
  $T(m)>0$ such that if $\cO$ is a hyperbolic periodic orbit
  for $X$ with
\begin{itemize}
\item angle between its stable and unstable directions
  bounded from below by $\theta$;
\item period larger than $T(m)$, and
\item the Linear Poincar\'e Flow along $\cO$ is not
  $m$-dominated,
\end{itemize}
then we can find a $\epsilon$-$C^1$-close vector field $Y$
such that $\cO$ is an elliptic periodic orbit for $Y^t$.
\end{lemma}

Conversely the absence of elliptic periodic orbits for all
nearby perturbations implies uniform bounds on hyperbolic
orbits with big enough period. This is an easy consequence
of the two previous Lemmas~\ref{le:smallangle}
and~\ref{le:nodomination} which we state for future
reference.

\begin{lemma}
  \label{le:farfromelliptic}
  Let $X\in \fX^{1}_{\mu}(M)$ and $\epsilon>0$ be given and
  set $\theta=\theta(\epsilon,X)$, $m=m(\epsilon,\theta)$
  and $T=T(m)$ given by Lemmas~\ref{le:smallangle}
  and~\ref{le:nodomination}. 

  Assume that all divergence-free vector fields $Y$ which
  are $\epsilon$-$C^1$-close to $X$ do not admit elliptic
  closed orbits. Then for every such $Y$ all closed orbits
  with period larger than $T$ are hyperbolic, $m$-dominated
  and with angle between its stable and unstable directions
  bounded from below by $\theta$.
\end{lemma}

\begin{proof}[Proof of Theorem~\ref{thr:open}]
  Let $\cP$ be the residual set given by Pugh's General
  Density Theorem in~\cite{PughRobin83}, that is $\cP$ is
  the family of all divergence-free vector fields $X$ such
  that $\Omega(X)$ is the closure of the set of periodic
  orbits, all of them hyperbolic or elliptic, and
  $\Omega(X)=M$ by the Poincaré Recurrence Theorem.

  We take any $X\in \fX^{1}_{\mu}(M)$ which is not
  approximated by an Anosov flow. Then by a small $C^1$
  perturbation we can assume that $X$ belongs to $\cP$ and
  that $X$ is still \emph{not} approximated by an Anosov
  flow.  We fix some open set $U$ and $\epsilon>0$.

  If some elliptic closed orbit of $X$ intersects $U$ there
  is nothing to prove, just set $Y=X$.  Otherwise we fix
  $\epsilon>0$ small and consider three cases:
\begin{enumerate}
\item [(A)] All closed orbits of $X$ which intersect $U$ are
  hyperbolic, and some of them has a small angle, less than
  $\theta=\theta(\epsilon,X)$ provided by
  Lemma~\ref{le:smallangle}, between the stable and unstable
  directions.
\item [(B)] All closed orbits of $X$ which intersect $U$ are
  hyperbolic, with angle between stable and unstable
  directions bounded from bellow by $\theta$, but some of
  them, with period larger than $T$, do not admit any
  $m$-dominated splitting for the Linear Poincar\'e Flow,
  where $m=m(\epsilon,\theta)$ and $T=T(m)$ are given by
  Lemma~\ref{le:nodomination}, and
  $\theta=\theta(\epsilon,X)$ was given as before by
  Lemma~\ref{le:smallangle}.
\item [(C)] All closed orbits of $X$ which intersect ${U}$
  and have period larger than $T$ are hyperbolic, with
  $m$-dominated splitting, and with the angle between the
  stable and unstable directions bounded from bellow by
  $\theta$, where $m=m(\epsilon,\theta)$ and $T=T(m)$ are given by
  Lemma~\ref{le:nodomination}, and
  $\theta=\theta(\epsilon,X)$ was given as before by
  Lemma~\ref{le:smallangle}.
\end{enumerate}

Case (A) implies the desired conclusion for some
zero-divergence vector field $Y$ $\epsilon$-$C^1$-close to
$X$ by Lemma~\ref{le:smallangle}.  Analogously for case (B)
by the choice of the bounds $m$, $T$ and by
Lemma~\ref{le:nodomination}.

Finally, we use Theorem~\ref{mthm:anosov-zero} to show that
if $X$ is in case (C) and we assume that every $C^1$-nearby
vector field $Y$ does not admit elliptic periodic orbits
through $U$, then we get a contradiction. This proves the
statement of Theorem~\ref{thr:open}.

\medskip

If $X$ is in case (C), then from
Lemma~\ref{le:farfromelliptic} we know that every periodic
orbit intersecting $U$, for every vector field $Y$
$\epsilon$-$C^1$-close to $X$, with period larger than $T$,
is hyperbolic with uniform bounds on $m$ and $\theta$.

From Theorem~\ref{mthm:anosov-zero}, since $X$ is not
approximated by an Anosov flow, there exists an
incompressible vector field $Y$, which is
$\epsilon/3$-$C^1$-close to $X$, admitting a full
$\mu$-measure subset $\Z$ where all Lyapunov exponents for
$Y$ are zero. Moreover we can assume that $Y$ is aperiodic,
that is the set of all periodic orbits has volume zero.

Let $\hat{U}\subset U$ be a measurable set with positive
measure. Let $R\subset \hat{U}$ be the set given by
Poincar\'e Recurrence Theorem (see e.g. \cite{Man87}) with
respect to $Y$. Then every $x\in R$ returns to $\hat{U}$
infinitelly many times under the flow $Y^t$ and is not a
periodic point. Denote by $\cT$ the set of positive return
times to $\hat{U}$ under $Y^t$.

Given $x\in \Z\cap R$ and $0<\delta<\log2/m$, there
exists $t_{x}\in\RR$ such that $$e^{-\delta
  t}<\|P_{Y}^{t}(x)\|<e^{\delta t} \text{ for every }t\geq
t_{x}.$$

Let us choose $\tau\in\cT$ such that $\tau>
\max\{t_{x},T\}$.

The $Y^t$-orbit of $x$ can be approximated for a very long
time $\tau>0$ by a periodic orbit of a $C^1$-close flow $Z$:
given $r,\tau>0$ we can find a
$\epsilon/3$-$C^1$-neigh\-bor\-hood $\cU$ of $Y$ in
$\fX^1_\mu(M)$, a vector field $Z\in\mathcal{U}$, a periodic
orbit $p$ of $Z$ with period $\ell$ and a map
$g:[0,\tau]\to[0,\ell]$ close to the identity such that
 \begin{itemize}
 \item $\dist\big( Y^t(x),Z^{g(t)}(p)\big)<r$ for all
   $0\le t\le \tau$;
 \item $Z=Y$ over $M\setminus\bigcup_{0\le t\le \ell}\big(
   B(p,r)\cap B(Z^t(p),r)\big)$.
 \end{itemize}
 This is Pugh's $C^1$ Closing Lemma adapted to the setting
 of conservative flows, see ~\cite{PughRobin83}.
 Letting $r>0$ be small enough we obtain also that
 \begin{align}\label{eq:almostid}
   e^{-\delta \ell}<\|P_{Z}^{\ell}(p)\|<e^{\delta \ell}
   \quad\text{with}\quad \ell> T. 
 \end{align}
 Now it is easy to see that $Z$ is $\epsilon$-$C^1$-close to
 $X$, so that the orbit of $p$ under $Z$ satisfies the
 conclusion of Lemma~\ref{le:farfromelliptic}. In particular
 we have that
 \begin{align*}
   \frac{\big\| DP_Z^m\mid N_x^s\big\|}
   {\big\| DP_Z^m\mid N_x^u\big\|}
   \le \frac12 \quad\text{for all}\quad
   x\in\cO_Z(p),
 \end{align*}
 for otherwise we would use Lemma~\ref{le:nodomination} and
 produce an elliptic periodic orbit for a flow
 $\epsilon$-$C^1$-close to $X$. Since the subbundles
 $N^{s,u}$ are one-dimensional we write $p_i:=Z^{im}(p)$ for
 $i=0,\dots, [\ell/m]$ with $[t]:=\max\{k\in\ZZ:k\le t\}$ and
 \begin{align}\label{eq:dominationend}
   \frac{\big\| DP_Z^\ell\mid N_p^s\big\|}
   {\big\| DP_Z^\ell\mid N_p^u\big\|}
   =
   \frac{\big\| DP_Z^{\ell-m\cdot[\ell/m]}\mid N_p^s\big\|}
   {\big\| DP_Z^{\ell-m\cdot[\ell/m]}\mid N_p^u\big\|}
   \cdot\prod_{i=0}^{[\ell/m]}
   \frac{\big\| DP_Z^m\mid N_{p_i}^s\big\|}
   {\big\| DP_Z^m\mid N_{p_i}^u\big\|}
   \le
   C(p,Z)\cdot\left(\frac12\right)^{[\ell/m]},
 \end{align}
 where $C(p,Z)=\sup_{0\le t\le m}\big(\| DP_Z^t\mid N_{p}^s\|\cdot \|
 DP_Z^t\mid N_{p}^u\|^{-1}\big)$ depends continuously on $Z$ in the
 $C^1$ topology.  There exists then a uniform bound on
 $C(p,Z)$ for all vector fields $Z$ which are $C^1$-close to
 $X$.

 We note that we can take $\ell>T$ arbitrarily big by
 letting $r>0$ be small enough in the above
 arguments. Therefore \eqref{eq:dominationend} ensures that
 $\|DP_Z^\ell(p)\|=\|DP_Z^\ell\mid N^u_p\|$ and also
 \begin{align*}
   \frac1{\ell}\log\big\| DP_Z^\ell\mid N_p^s\big\|
   \le \frac1{\ell}\log C(p,Z)
   +\frac{[\ell/m]}{\ell}\log\frac12
   +\frac1{\ell}\log\big\| DP_Z^\ell\mid N_p^u\big\|.
 \end{align*}
 Moreover since  $Z$ is volume preserving we have that the
 sum of the Lyapunov exponents along $\cO_Z(p)$ is zero,
 that is (we recall that $\ell$ is the period of $p$)
 \begin{align*}
   \frac1{\ell}\log\| DP_Z^\ell\mid N_p^s\|=-\frac1{\ell}\log\|
   DP_Z^\ell\mid N_p^u\|.
 \end{align*}
 The constants in~\eqref{eq:dominationend} do not depend on
 $\ell$ so taking the period very big we deduce that
 \begin{align*}
   \frac1{\ell}\log\| DP_Z^\ell(p)\| \ge
   \frac1m\log2>\delta.
 \end{align*}
 This contradicts~\eqref{eq:almostid} and completes the
 proof of Theorem~\ref{thr:open}.
\end{proof}


\section{Dominated splitting and regularity}
\label{sec:singul-hyperb}

Here we prove that positive volume regular invariant subsets
with dominated splitting cannot admit singularities in its
closure and thus are essentially uniformly hyperbolic
sets. This result will be used to prove
Theorem~\ref{thm:there-exists-generic}.

We denote by $\fX^{1+}(M)$ the set of all $C^1$ vector
fields $X$ whose derivative $DX$ is H\"older continuous with
respect to the given Riemannian norm, and we say that
$X\in\fX^{1+}(M)$ is of class $C^{1+}$. We clearly have
$$\fX^1(M)\supset\fX^{1+}(M)\supset\fX^r(M), \quad\text{for every $r\ge2$}.$$

\begin{proposition}
  \label{pr:DDLPF-singhyp}
  Let $X\in\fX^{1+}_\mu(M)$ be given. Assume that
  $\Lambda$ is a regular $X^t$-invariant subset of $M$ with
  positive volume and admitting a dominated splitting. Then
  the closure $A$ of the set of Lebesgue density points of
  $\Lambda$ does not contain singularities.
\end{proposition}

We recall that a compact invariant subset $\Lambda$ of
$X\in\fX_{\mu}^1(M)$ is (uniformly) \emph{hyperbolic}
if $$T_\Lambda M=E\oplus [X] \oplus G$$ is a continuous
$DX^t$-invariant splitting with the sub-bundle $E\neq\vec{0}$
uniformly contracted and the sub-bundle $G\neq\vec{0}$ uniformly
expanded by $DX^t$ for $t>0$.

According to ~\cite[Lemma 2.4]{bessa} a compact
invariant set without singularities of a $C^1$
three-dimensional vector field admitting a dominated
splitting for the Linear Poincaré Flow is a uniformly
hyperbolic set. Then we obtain the following.
\begin{corollary}
  \label{cor:DDFLP-unifhyp}
  Let $X\in\fX^{1+}_\mu(M)$ and
  $\Lambda$ be a regular $X^t$-invariant subset of $M$ with
  positive volume and admitting a dominated splitting. Then
  the closure $A$ of the set of Lebesgue density points of
  $\Lambda$ is a hyperbolic set.
\end{corollary}

This implies in particular that there are neither
singular-hyperbolic sets (e.g. Lorenz-like sets or
singular-horseshoes) nor partially hyperbolic sets (see
e.g. \cite{BDV2004} or \cite{MPP04} for the definitions)
with positive volume for $C^{1+}$ incompressible flows on
three-dimensional manifolds. A similar conclusion for
singular-hyperbolic sets was obtained by Arbieto-Matheus
in~\cite{ArbiMath} but assuming that the invariant compact
subset is robustly transitive.

The proof of Proposition~\ref{pr:DDLPF-singhyp} is divided
into several steps, which we state and prove as a sequence
of lemmas in the following subsections.

\subsection{Bounded angles, eigenvalues and Lorenz-like singularities}
\label{sec:bounded-angles-eigen}

Denote by $D(\Lambda)$ the subset of the Lebesgue density
points of $\Lambda$, that is, $x\in D(\Lambda)$ if
$x\in\Lambda$ and
\begin{align*}
  \lim_{r\to0^+}\frac{\mu(\Lambda\cap B(x,r))}{\mu(B(x,r))}=1.
\end{align*}
Is is well known (see e.g. \cite{Ru21} or \cite{Munroe})
that almost every point of a measurable set are Lebesgue
density points, that is $\mu(\Lambda\setminus
D(\Lambda))=0$. Moreover since every nonempty open subset of
$M$ has positive $\mu$-measure, we see that $D(\Lambda)$ is
contained in the closure of $\Lambda$.

Assume that $\Lambda$ is a $X^t$-invariant set without
singularities such that $\mu(\Lambda)>0$ and write $A$ for
the closure of $D(\Lambda)$ in what follows. Note that $A$
is contained in the closure of $\Lambda$.

\begin{lemma}\label{le:DDFLP-noelliptic}
  Suppose that the Linear Poincar\'e Flow over $\Lambda$ has a
  dominated splitting for $X$. Then there exist a
  neighborhood $V$ of $\Lambda$, a neighborhood $\U$ of $X$
  in $\fX^1(M)$ (not necessarily contained in the space of
  conservative flows) and $\eta>0$ such that for every
  $Y\in\U$, every periodic orbit contained in $U$ is
  hyperbolic of saddle type and its eigenvalues
  $\lambda_1$ and $\lambda_2$ satisfy $\lambda_1<-\eta$
  and $\lambda_2>\eta$. Moreover the angle between the
  unstable and stable directions of these periodic
  orbits is greater than $\eta$.
\end{lemma}

\begin{proof}
  The Dominated Splitting for the Linear Poincar\'e Flow
  extends by continuity to every \emph{regular} orbit $\cO$
  which remains close to $\Lambda$ for a $C^1$ nearby flow
  $Y$, this is Lemma~\ref{le:persistdom}. The domination
  implies that the eigenvalues $\lambda_1\le\lambda_2$ of
  $\cO$ satisfy $\lambda_1+2\kappa\le\lambda_2$ for some
  $\kappa>0$ which only depends on the domination constant
  of $\Lambda$.  Since the flow $Y$ is close to being
  conservative, we have $|\lambda_1+\lambda_2|\le\epsilon$,
  where we can take $\epsilon<\kappa/2$ just by letting $Y$
  be in a small $C^1$-neighborhood of $X$.

  Thus we have $-\lambda_2-\epsilon\le\lambda_1$ which
  implies
  $-\lambda_2-\epsilon+2\kappa\le\lambda_1+2\kappa\le\lambda_2$
  and so $2\lambda_2\ge2\kappa-\epsilon>0$ on the one hand.
  On the other hand $\lambda_1\le\epsilon-\lambda_2$ implies
  $\lambda_1\le\epsilon-(\kappa-\epsilon/2)=3\epsilon/2-\kappa<0$.

  Hence there exists $\eta>0$, independent of $Y$ in a $C^1$
  neighborhood of $X$, and independent of the periodic orbit
  $\cO$ of $Y$ in a neighborhood of $\Lambda$, such that
  $\lambda_1<-\eta$ and $\lambda_2>\eta$, as stated.

  For the angle bound we argue by contradiction as
  in~\cite{MPP04}: assume there exists a sequence of flows
  $Y_n\xrightarrow[n\to+\infty]{C^1}X$ and of periodic
  orbits $\cO_n$ of $Y_n$ contained in the neighborhood $V$
  of $\Lambda$ such that the angle $\alpha_n$ between the
  unstable subspace and the stable direction
  satisfies $\alpha_n\xrightarrow[n\to+\infty]{}0$.

  Then as in the proof of~\cite[Theorem 3.6]{MPP04}
  (or~\cite[Theorem 3.31]{AraPac07}) we can find (through a
  flow version of Frank's Lemma, see~\cite{Fr71}
  and~\cite[Appendix]{AraPac07}) an arbitrarily small $C^1$
  perturbation $Z_n$ of $Y_n$, for all big enough $n\ge1$,
  sending the stable direction close to the unstable
  direction along the periodic orbit, such that the orbit of
  $\cO_n$ becomes a sink or a source for $Z_n$. This
  contradicts the first part of the statement of the lemma.
\end{proof}

We say that a singularity $\sigma$ is \emph{Lorenz-like} for
$X$ if $DX(\sigma)$ has three real eigenvalues
$\lambda_2\le\lambda_1\le\lambda_3$ satisfying
$\lambda_2<\lambda_3<0<-\lambda_3<\lambda_1$.

\begin{lemma}
  \label{le:singLorenz}
  Assume that $X\in\fX^1_\mu(M)$ is such that all
  singularities are hyperbolic with no ressonances
  (real eigenvalues are all distinct).  Then the
  singularities $S(X)\cap A$ are all Lorenz-like for $X$ or
  for $-X$.
\end{lemma}

\begin{remark}
  \label{rmk:opendense}
  The assumptions of the lemma above hold true for an open
  and dense subset of all $C^r$ vector fields, both volume
  preserving or not.
\end{remark}

\begin{proof}
  Fix $\sigma$ in $S(X)\cap A$ if this set is nonempty
  (otherwise there is nothing to prove). By assumption on
  $X$ we known that $\sigma$ is hyperbolic. As
  in~\cite{MPP04} we show first that $\sigma$ has only real
  eigenvalues. For otherwise we would get a conjugate pair
  of complex eigenvalues $\omega, \overline\omega$ and a
  real one $\lambda$ and, by reversing time if needed, we
  can assume that $\lambda<0<\text{Re}(\omega)$. Since
  $\mu(A)>0$ there are infinitely many distinct orbits of
  $\Lambda$ passing through every given neighborhood of
  $\sigma$, for each regular orbit of a flow is a regular
  curve, and so does not fill volume in a three-dimensional
  manifold.
  
  Using the Connecting Lemma of Hayashi adapted to
  conservative flows (see e.g.~\cite{WX00}) we can find a
  $C^1$-close flow $Y$ preserving the same measure $\mu$
  with a saddle-focus connection associated to the
  continuation $\sigma_Y$ of the singularity $\sigma$.  By a
  small perturbation of the vector field we can assume that
  $Y$ is of class $C^\infty$ and still $C^1$-close to $X$
  (see e.g.~\cite{zup79}).

  We can now unfold the saddle-focus connection as in
  \cite{BiraShil92} to obtain a periodic orbit with all
  Lyapunov exponents equal to zero (an elliptic closed
  orbit) for a $C^1$-close flow and near $A$. This
  contradicts Lemma~\ref{le:DDFLP-noelliptic}, since such
  orbit will be contained in a neighborhood of $\Lambda$.
  This shows that complex eigenvalues are not allowed for
  any singularity in $A$.

  Let then $\lambda_2\le\lambda_3\le\lambda_1$ be the
  eigenvalues of $\sigma$. We have $\lambda_2<0<\lambda_1$
  because $\sigma$ is hyperbolic. The preservation of volume
  implies that $\lambda_2=-(\lambda_1+\lambda_3)<0$ so that
  $-\lambda_3<\lambda_1$. We have now two cases:
  \begin{description}
  \item[$\lambda_3<0$] this implies
    $\lambda_2<\lambda_3<0<-\lambda_3<\lambda_1$ by the
    non-resonance assumption, and $\sigma$ is Lorenz-like
    for $X$;
  \item[$\lambda_3>0$] since
    $\lambda_1=-(\lambda_2+\lambda_3)>0$ the non-resonance
    assumption ensures that
    $\lambda_2<-\lambda_3<0<\lambda_3<\lambda_1$, so
    $\sigma$ is Lorenz-like for $-X$.
  \end{description}
  The proof is complete.
\end{proof}

\subsection{Invariant manifolds of a positive volume set
  with dominated splitting for the Linear Poincaré Flow}
\label{sec:invari-manifolds-pos}

\subsubsection{Invariant manifolds and (non-uniform) hyperbolicity}
\label{sec:invari-manifolds-non}

An embedded disk $\gamma\subset M$ is a (local) {\em
  strong-unstable manifold}, or a {\em strong-unstable
  disk}, if $\dist(X^{-t}(x),X^{-t}(y))$ tends to zero
exponentially fast as $t\to+\infty$, for every
$x,y\in\gamma$. In the same way $\gamma$ is called a (local)
{\em strong-stable manifold}, or a {\em strong-stable disk},
if $\dist(X^{t}(x),X^{t}(y))\to0$ exponentially fast as
$n\to+\infty$, for every $x,y\in\gamma$. It is well-known
that every point in a uniformly hyperbolic set possesses a
local strong-stable manifold $W_{loc}^{ss}(x)$ and a local
strong-unstable manifold $W_{loc}^{uu}(x)$ which are disks
tangent to $E_x$ and $G_x$ at $x$ respectively with
topological dimensions $d_E=\dim(E)$ and $d_G=\dim(G)$
respectively. Considering the action of the flow we get the
(global) \emph{strong-stable manifold}
$$W^{ss}(x)=\bigcup_{t>0}
X^{-t}\Big(W^{ss}_{loc}\big(X^t(x)\big)\Big)$$
and the
(global) \emph{strong-unstable manifold}
$$W^{uu}(x)=\bigcup_{t>0}X^{t}\Big(W^{uu}_{loc}\big(X^{-t}(x)\big)\Big)$$
for every point $x$ of a uniformly hyperbolic set. Similar
notions are defined in a straightforward way for
diffeomorphisms.  These are immersed submanidfolds with the
same differentiability of the flow or the diffeomorphism.
In the case of a flow we also consider the \emph{stable
  manifold} $W^s(x)=\cup_{t\in\RR} X^{t}\big(W^{ss}(x)\big)$
and \emph{unstable manifold}
$W^u(x)=\cup_{t\in\RR}X^{t}\big(W^{uu}(x)\big)$ for $x$ in a
uniformly hyperbolic set, which are flow invariant.

We note that these notions are well defined for a hyperbolic
periodic orbit, since this compact set is itself a
hyperbolic set.

Now we observe that since $A$ has positive volume, 
the dominated splitting of the Linear Poincaré Flow 
implies that the Lebesgue measure $\mu_A$ normalized and
restricted to $A$ is a (non-uniformly) \emph{hyperbolic
  invariant probability measure}, see
e.g. \cite{barreira-pesin}: every Lyapunov exponent of
$\mu_A$ is non-zero, except along the direction of the
flow. Indeed, (recall the arguments in the proof of
Lemma~\ref{le:DDFLP-noelliptic}) the Lyapunov exponents
$\lambda_1\le 0 \le \lambda_2$ along every Oseledet's
regular orbit satisfy $\lambda_1+\lambda_2=0$ since the flow
is incompressible, and for every Oseledet's regular orbit in
$\Lambda$ (a non-empty set because $\Lambda$ has positive
volume) the exponents also satisfy
$\lambda_1+2\kappa\le\lambda_2$ for some $\kappa>0$
depending only on the domination strength --- in particular
$\kappa$ \emph{does not depend on the orbit chosen inside
  $\Lambda$}. Thus there exists $\eta>0$ such that
$\lambda_2=-\lambda_1>\eta$ along every Oseledet's regular
orbit inside $\Lambda$.

\emph{Assuming from now on that $X\in\fX^{1+}(M)$} we have,
according to the non-uniform hyperbolic theory (see
\cite{Pe76,Pe77,barreira-pesin}), that there are smooth
strong-stable and strong-unstable disks tangent to the
directions corresponding to negative and positive Lyapunov
exponents, respectively, at $\mu_A$ almost every point. The
sizes of these disks depend measurably on the point as well
as the rates of exponential contraction and expansion. We
can define as before the strong-stable, strong-unstable,
stable and unstable manifolds at $\mu_A$ almost all points.

In addition, since $\mu$ is a smooth invariant measure, we
can use \cite[Theorem 11.3]{barreira-pesin2} and conclude
that there are at most countably many ergodic components of
$\mu_A$. \emph{Therefore we assume from now on that $\mu_A$
  is ergodic without loss of generality.}

In addittion, hyperbolic smooth ergodic invariant
probability measures for a $C^{1+}$ dynamics are in the
setting of Katok's Closing Lemma, see~\cite{katok80}
or~\cite[Section 15]{barreira-pesin2}. In particular we have
that the support of $\mu_A$ is contained in the closure of
the closed orbits inside $A$
\begin{align}\label{eq:perdense}
  \supp(\mu_A)\subset\overline{\per(X)\cap A},
\end{align}
where the periodic points in our setting are all hyperbolic
by Lemma~\ref{le:DDFLP-noelliptic}.

\subsubsection{Almost all invariant manifolds are contained in $A$}
\label{sec:almost-all-invari}

Now we adapt the arguments in \cite{BochVi04} to our setting
to deduce the following. Let $\mu_u$ and $\mu_s$ denote the
measure induced on (strong-)unstable and (strong-)stable
manifolds by the Lebesgue volume form $\mu$.

\begin{lemma}
  \label{le:DDFLP-Wsinside}
  For $\mu_A$ almost every $x$ the corresponding invariant
  manifolds satisfy
  \begin{align*}
    \mu_s\big(W^{ss}(x) \setminus A\big)=0 \quad\text{and}\quad
    \mu_u\big(W^{uu}(x)\setminus A\big)=0
  \end{align*}
  that is, the invariant manifolds are $\mu_{u,s}\bmod0$
  contained in $A$.
\end{lemma}

In addition, since $A$ is closed and every open subset of
either $W^{ss}(x)$ or $W^{uu}(x)$ has positive $\mu_s$ or
$\mu_u$ measure, respectively, then we see that in fact
\begin{align}\label{eq:WscontainedA}
  W^{ss}(x)\subset A \quad\text{and}\quad W^{uu}(x)\subset A
  \quad\text{for  } \mu-\text{almost every  } x.
\end{align}

To prove Lemma~\ref{le:DDFLP-Wsinside} we need a bounded
distortion property along invariant manifolds which is
provided by \cite[Theorems 11.1 \&
11.2]{barreira-pesin2}. To state this properly we need the
notion of hyperbolic block for a hyperbolic invariant
probability measure.

\subsubsection{Hyperbolic blocks and bounded distortion along invariant manifolds}
\label{sec:hyperb-blocks-bounde}

The measurable dependence of the invariant manifolds on the
base point means that for each $\kappa\in\NN$ we can find a
compact \emph{hyperbolic block} $\cH(\kappa)$ and positive
numbers $C_x$ satisfying
\begin{itemize}
\item $\dist(X^t(y),X^t(x))\le C_x
  e^{-t\tau}\cdot\dist(y,x)$ for all $t>0$ and $y\in
  W_{loc}^{ss}(x)$, and analogously for $y\in
  W_{loc}^{uu}(x)$ exchanging the sign of $t$;
\item $C_x\le\kappa$ and $\tau_x\ge\kappa^{-1}$ for every
  $x\in \cH(\kappa)$;
\item $\cH(\kappa)\subset\cH(\kappa+1)$ for all $k\ge1$ and
  $\mu_A\big(\cH(\kappa)\big)\to1$ as $\kappa\to+\infty$;
\item the $C^1$ strong-stable and strong-unstable disks
  $W^{ss}_{loc}(x)$ and $W^{uu}_{loc}(x)$ vary continuously
  with $x\in\cH(\kappa)$ (in particular the sizes of these
  disks and the angle between them are uniformly bounded
  from zero for $x$ in $\cH(\kappa)$).
 \end{itemize}

Now we have the bounded distortion property.

 \begin{theorem}{\cite[Theorems 11.1 \&
     11.2]{barreira-pesin2}}
   \label{thm:boundedistortion}
   Fix $\kappa\in\NN$ such that $\mu_A(\cH(\kappa))>0$. Then
   the function
   \begin{align*}
     h^s(x,y):=\prod_{i\ge0} \frac{\big|\det Df\mid
       E^s(f^i(x))\big|}{\big|\det Df\mid E^s(f^i(y))\big|}  
   \end{align*}
   is H\"older-continuous for every $x\in\cH(\kappa)$ and
   $y\in W^{ss}_{loc}(x)$, where $f:=X^1$ is the time-$1$
   map of the flow $X^t$ and $E^s$ is the direction
   corresponding to negative Lyapunov exponents. 

   An analogous statement is true for a function $h^u$ on
   the unstable disks in $\cH(\kappa)$ exchanging $E^s$ with
   the direction $E^u$ corresponding to positive Lyapunov
   exponents and reversing the sign of $i$ in the product
   $h^s$ above.
 \end{theorem}

 Note that since $\cH(\kappa)$ is compact, there exists
 $0<h_\kappa<\infty$ such that $\max\{h^u,h^s\}\le h_\kappa$ on
 $\cH(\kappa)$.

\subsubsection{Recurrent and Lebesgue density points}
\label{sec:recurr-lebesg-densit}

We are now ready to start the proof of
Lemma~\ref{le:DDFLP-Wsinside}.

Let us take a strong-unstable disk $W^{uu}(x)$ satisfying
simultaneously
\begin{itemize}
\item $x\in\cH(\kappa)$,
\item $\mu_u\big(W^{uu}(x)\cap
A\big)>0$ and 
\item $x$ is a $\mu_u$ density point of
$W^{uu}(x)\cap A$. 
\end{itemize}

For this it is enough to take $\kappa$
big enough since by the absolute continuity of the foliation
of strong-unstable disks a positive volume subset, as
$\cH(\kappa)$, must intersect almost all strong-stable disks
on a subset of $\mu_u$ positive measure, see
e.g. \cite{PS89}.

Using the Recurrence Theorem we can also assume without loss
of generality that $x$ is recurrent inside $\cH(\kappa)$,
that is, there exists a strictly increasing sequence of
integers $n_1<n_2<\dots$ such that
\begin{align*}
 x_k:= f^{n_k}(x)\in\cH(\kappa) \quad\text{for all}\quad k\in\NN
  \quad\text{and}\quad x_k \xrightarrow[k\to\infty]{}x.
\end{align*}
Therefore we can consider the disk
$W_k=f^{-n_k}\big(W_{loc}^{uu}(x_k)\big)$. Observe that
$W_k\subset W^{uu}_{loc}(x)$ is a neighborhood of $x$ and
since the sizes of the strong-unstable disks on
$\cH(\kappa)$ are uniformly bounded we see that
$\diam\big(W_k)\to0$ exponentially fast as $k\to+\infty$.

Now $W^{uu}_{loc}(x)$ is one-dimensional in our setting and
thus the shrinking of $W_k$ to $x$ together with the
$f$-invariance of $A$ are enough to ensure
\begin{align*}
  \frac{\mu_u\Big(f^{-n_k}\big(W_{loc}^{uu}(x_k)\setminus
    A\big)\Big)}{\mu_u\Big(f^{-n_k}\big(W_{loc}^{uu}(x_k)\big)\Big)}
  =\frac{\mu_u\big(W_k\setminus A\big)}{\mu_u(W_k)}
  \xrightarrow[k\to\infty]{} 0.
\end{align*}
Finally the bounded distortion given by
Theorem~\ref{thm:boundedistortion} implies
\begin{align*}
  \frac{\mu_u\Big(f^{-n_k}\big(W_{loc}^{uu}(x_k)\setminus
    A\big)\Big)}{\mu_u\Big(f^{-n_k}\big(W_{loc}^{uu}(x_k)\big)\Big)}
  &=
  \frac{\int_{W_{loc}^{uu}(x_k)\setminus A}|\det
    Df^{-n_k}\mid E^u(z)|\,d\mu_u(z)}{\int_{W_{loc}^{uu}(x_k)}|\det
    Df^{-n_k}\mid E^u(z)|\,d\mu_u(z)}
  \\
  &\ge
  \frac1{h^u_\kappa} \cdot  \frac{\mu_u\big(W_{loc}^{uu}(x_k)\setminus
    A\big)}{\mu_u(W_{loc}^{uu}(x_k))},
\end{align*}
which means that
\begin{align*}
  \frac{\mu_u\big(W_{loc}^{uu}(x_k)\setminus
    A\big)}{\mu_u(W_{loc}^{uu}(x_k))}
  \le
  h_\kappa\cdot\frac{\mu_u\big(W_k\setminus A\big)}{\mu_u(W_k)}
\end{align*}
for all $k\ge1$. Hence we get
$\mu_u(W_{loc}^{uu}(x)\setminus A)=0$ by the choice of $x_k$
and the continuous dependence of the strong-unstable disks
on the points of the hyperbolic block $\cH(\kappa)$.  The
argument for the stable direction is the same. Since the
points of a full $\mu_A$ measure subset have all the
properties we used, this concludes the proof of
Lemma~\ref{le:DDFLP-Wsinside} and of the
property~\eqref{eq:WscontainedA}.

\subsubsection{Dense invariant manifolds of a periodic orbit}
\label{sec:invari-manifolds-per}

Now we use the density of periodic points in $A$ (property
\eqref{eq:perdense}). Consider again a hyperbolic block
$\cH(\kappa)$ with a big enough $\kappa\in\NN$ such that
$\mu_A(\cH(\kappa))>0$. For any given $x\in\cH(\kappa)$
and $\delta>0$ there exists a hyperbolic periodic orbit
$\cO(p)$ intersecting $B(x,\delta)$. Because the sizes and
angles of the stable and unstable disks of points in
$\cH(\kappa)$ are uniformly bounded away from zero, we can
ensure that we have the following transversal intersections
\footnote{Recall the difference between $W^{uu}(p)$ and $W^u(p)$ etc
in the flow setting.}
\begin{align*}
  W^u(p)\pitchfork W^s(x)\neq\emptyset\neq W^s(p)\pitchfork W^u(x).
\end{align*}
This together with the Inclination Lemma implies that
\begin{align}
  \label{eq:WsperiodicInside}
  \overline{W^{u}(p)}=\overline{W^u(x)}\subset A
  \quad\text{and}\quad
  \overline{W^{s}(p)}=\overline{W^s(x)}\subset A.
\end{align}
Moreover since we can pick any $x\in\cH(\kappa)$ we can
assume without loss that $x$ has a dense orbit in $A$ (since
we took $\mu_A$ to be ergodic) and then we can
strengthen~\eqref{eq:WsperiodicInside} to: there exists a
periodic orbit $\cO(p)$ inside $A$ such that
\begin{align}
  \label{eq:WperiodicDense}
  \overline{W^{u}(p)}= A
  \quad\text{and}\quad
  \overline{W^{s}(p)}= A.
\end{align}

\subsubsection{Absence of singularities in $A$}
\label{sec:absence-singul-a}

We recall that the \emph{alpha-limit set} of a point $p\in
M$ with respect to the flow $X$ is the set $\alpha(p)$ of
all limit points of $X^{-t}(p)$ as $t\to+\infty$. Likewise
the \emph{omega-limit set} is the set $\omega(p)$ of limit
points of $X^t(p)$ when $t\to+\infty$. Both these sets are
flow-invariant.

Using property~\eqref{eq:WperiodicDense} we consider, on the
one hand, the
invariant compact subset of $A$ given by
\begin{align*}
  L=\alpha_X(W^{ss}(p))
\end{align*}
the closure of the accumulation points of backward orbits of
points in the strong-stable manifold of the periodic orbit
$\cO(p)$. By~\eqref{eq:WperiodicDense} we have $L=A$.  On
the other hand, considering $N=\omega_X(W^{uu}(p))$ we
likewise obtain that $N=A$.

Let us assume that $\sigma$ is a singularity contained in
$A$. By Lemma~\ref{le:singLorenz} $\sigma$ is either
Lorenz-like for $X$ or Lorenz-like for $-X$. 

In the former case, we would get $W^{ss}(\sigma)\subset A$
because any compact part of the strong-stable manifold of
$\sigma$ is accumulated by backward iterates of a small
neighborhood $\gamma$ inside $W^{ss}(x)$. Here we are using
that the contraction along the strong-stable manifold, which
becomes an expansion for negative time, is uniform.  In the
latter case we would get $W^{uu}(\sigma)\subset A$ by a
similar argument reversing the time direction.

We now explain that each one of these possibilities leads to
a contradiction with the dominated splitting of the Linear
Poincaré Flow on the regular orbits of $A$, following an
argument in~\cite{MPP04}. It is enough to deduce a
contradiction for a Lorenz-like singularity for $X$, since
the other case reduces to this one through a time inversion.

If $W^{ss}(\sigma)\cap A\setminus\{\sigma\}\supset\{y\}$ for
some point $y\in A$ and for some singularity $\sigma\in A$,
then we have countably distinct regular orbits of $\Lambda$
accumulating on $y\in W^{ss}(\sigma)$ (by the definition of
$A$) and on a point $q\in W^u(\sigma)$  (by the dynamics of
the flow near $\sigma$). 

Applying the Connecting Lemma, we obtain a saddle-connection
associated to the continuation of $\sigma$ for a $C^1$-close
vector field $Y$, known as ``orbit-flip'' connection, that
is, there exists a homoclinic orbit $\Gamma$ associated to
$\sigma_Y$ such that $W^{cu}(\sigma_Y)$ intersects
$W^s(\sigma_Y)$ transversely along $\Gamma$, i.e.
$\Gamma=W^{cu}(\sigma_Y)\pitchfork W^s(\sigma_Y)$, and also
$\Gamma\cap W^{ss}(\sigma_Y)\neq\emptyset$.

These connections can be $C^1$ approximated by
``inclination-flip'' connections for another $C^1$ nearby
vector field $Z$, \emph{not necessarily conservative}, see
e.g.~\cite{MPa2,AraPac07}. This means that the continuation
$\sigma_Z$ of the singularity has an associated homoclinic
orbit $\gamma$ such that $W^{cu}(\sigma_Z)$ intersects
$W^s(\sigma_Z)$ along $\gamma$ \emph{but not transversely},
and $\gamma\cap W^{ss}(\sigma_Z)=\emptyset$.

However the presence of ``inclination-flip'' connections is
an obstruction to the dominated decomposition of the Linear
Poincar\'e Flow for nearby regular orbits. This contradicts
Lemma~\ref{le:persistdom} and concludes the proof of
Proposition~\ref{pr:DDLPF-singhyp}.

\section{Uniform hyperbolicity}
\label{sec:uniform-hyperb}

Here we conclude the proof of
Theorem~\ref{thm:there-exists-generic}, showing that proper
invariant hyperbolic subsets of a $C^{1+}$ incompressible
flow cannot have positive volume.

\begin{proposition}
  \label{pr:unifhyp-volzero}
  Let $A$ be a compact invariant hyperbolic subset for
  $X\in\fX^{1+}_{\mu}(M)$.  Then either $\mu(A)=0$ or else
  $X$ is an Anosov flow and $A=M$.
\end{proposition}

The proof of Proposition~\ref{pr:unifhyp-volzero} is given
as a sequence of intermediate results along the rest of this
section.  Assuming this result we easily have the following.

\begin{proof}[Proof of
  Theorem~\ref{thm:there-exists-generic}]
  From Corollary~\ref{cor:DDFLP-unifhyp} we have that a
  regular invariant subset with positive volume with
  dominated splitting for the Linear Poincaré Flow admits a
  positive volume subset which is hyperbolic. Therefore the
  flow of $X$ is Anosov from
  Proposition~\ref{pr:unifhyp-volzero}.
\end{proof}

\subsection{Positive volume hyperbolic sets and
  conservative Anosov flows}
\label{sec:positive-volume-hype}

We start the proof by recalling the notion of partial
hyperbolicity.

Let $\Lambda$ be a compact invariant subset for a $C^1$ flow
on a compact boundaryless manifold $M$ with dimension at
least $3$. We say that $\Lambda$ is \emph{partially
  hyperbolic} if there are a continuous invariant tangent
bundle decomposition $T_\Lambda M=E^s\oplus E^c$ and
constants $\lambda,K>0$ such that for all $x\in\Lambda$ and
for all $t\ge0$
\begin{itemize}
\item
{\em $E^s$ dominates $E^c$}:
$
\| DX^t(x)\mid E^s_x\|\cdot \| DX^{-t}\mid E^c_{X^t(x)}\|
\leq  K e^{-\lambda t}
$
\item
$E^s$ is uniformly contracting:
$
  \|DX^t\mid E^s_x\|\le K e^{-\lambda t}.
$
\end{itemize}
We note that for a partially hyperbolic set of a flow
\emph{the flow direction must be contained in the central
  bundle.}

Now we recall the following result.

  \begin{theorem}{\cite[Theorem 2.2]{AAPP}}
    \label{thm:a2p2-2.2}
    Let \( f: M\to M \) be a \( C^{1+} \) diffeomorphism
    and let $\Lambda\subset M$ be a partially hyperbolic set
    with positive volume.  Then $\Lambda$ contains a
    strong-stable disk.
  \end{theorem}

Now we can use an argument similar to the one presented in
Subsection~\ref{sec:absence-singul-a}.

  \begin{lemma}
    \label{le:cons-disk-Wss-in}
    Let $X\in\fX^{1+}_{\mu}(M)$ and $\Lambda$ be a compact
    invariant partially hyperbolic subset containing a
    strong-stable disk $\gamma$. Then
    $L=\alpha_X(\gamma)=\{\alpha(z):z\in\gamma\}$ contains all
    stable disks through its points.
  \end{lemma}

  \begin{proof}
    The partial hyperbolic assumption on $A$ ensures that
    every one of its points has a strong-stable
    manifold. Moreover
\begin{equation}\label{eq.ssinside}
W^{ss}(z)\subset\Lambda \quad\text{for every } z\in \alpha(\gamma),
\end{equation}
since any compact part of the strong-stable manifold of $z$
is accumulated by backward iterates of any small
neighborhood of $x\in\gamma$ inside $W^{ss}(x)$. Here we are
using that the contraction along the strong-stable manifold,
which becomes an expansion for negative time, is uniform.
  \end{proof}

  \begin{proof}[Proof of
  Proposition~\ref{pr:unifhyp-volzero}]
  Let $A$ be a hyperbolic subset for $X\in\fX^1_\mu(M)$ with
  $\mu(A)>0$.  From Lemma~\ref{le:cons-disk-Wss-in} we have
  that $L=\alpha(\gamma)$ satisfies
  $W^{ss}(L)=\{W^{ss}(z):z\in L\}\subset L$. This implies
  $W^{s}(L)= L$ by invariance.
    
  Consider now $W^u(L)=\{W^u(z):z\in L=W^s(L)\}$. This
  collection of unstable leaves crossing the stable leaves
  of $L$ forms a neighborhood of $L$.  But $U=W^u(L)$ being
  a neighborhood of $L$ means that $L$ is a repeller: for
  $w\in U$ we have
  $\dist\big(X^{-t}(w),L\big)\xrightarrow[t\to+\infty]{}0$.

    This contradicts the preservation of the volume form
    $\mu$, unless $L$ is the whole of $M$. Thus
    $M=L\subset A$ and $X$ is Anosov.
  \end{proof}


\def\cprime{$'$}


\end{document}